\documentclass[final,notitlepage,12pt,reqno]{amsart}
\usepackage{graphicx}
\usepackage{calc}
\usepackage{url}
\newlength{\depthofsumsign}
\makeatletter
\let\I\@undefined
\makeatother
\usepackage{geometry}
\usepackage{indentfirst}
\usepackage[normalem]{ulem}
\usepackage{float}
\usepackage{amsthm}

\usepackage{enumitem}[2011/09/28]
\setenumerate{align=left, leftmargin=0pt,labelsep=.5em, labelindent=0\parindent,listparindent=\parindent,itemindent=*}
\usepackage{array}
\usepackage{empheq}
\usepackage{natbib}
\usepackage[all]{xy}

\usepackage{multirow}

\usepackage{caption}[2012/02/19]

\usepackage{longtable,lscape}

\usepackage{fouriernc}
\usepackage{fourier}
\usepackage{amssymb,bm,amsmath}
\usepackage{amsfonts}

\usepackage[OT2,T1,T2A]{fontenc}

\usepackage[english,french,german,,romanian,russian]{babel}
\usepackage{appendix}

\DeclareMathOperator{\Int}{Int}

\DeclareMathOperator{\D}{d}
\DeclareMathOperator{\I}{Im}
\DeclareMathOperator{\R}{Re}

\def\XXint#1#2#3{{\setbox0=\hbox{$#1{#2#3}{\int}$}
     \vcenter{\hbox{$#2#3$}}\kern-.5\wd0}}

\bibpunct{[}{]}{,}{n}{}{;}

\def\eor{\hfill$ \square$}

\theoremstyle{plain}
\newtheorem{theorem}{Theorem}[section]

\newtheorem{lemma}[theorem]{Lemma}

\newenvironment{remark}[1][Remark]{\begin{trivlist}
\item[\hskip \labelsep {\bfseries #1}]}{\end{trivlist}}

\theoremstyle{definition}

\numberwithin{equation}{section}

\usepackage{color}

\begin{document}

\selectlanguage{english}
\title[Ramanujan Series, Epstein Zeta Functions]{Ramanujan Series for Epstein Zeta Functions}
\author{Yajun Zhou
}
\address{
}
\email{\vspace{2em}yajunz@math.princeton.edu}
\date{\today}

\maketitle

\begin{abstract}
    In the spirit of Ramanujan, we derive exponentially fast convergent series for Epstein zeta functions  $ E^{\varGamma_0(N)}(z,s)$ on the Hecke congruence groups $ \varGamma_0(N),N\in\mathbb Z_{>0}$, where $z$ is an arbitrary point in the upper half-plane  $ \mathfrak H$, and $s\in\mathbb Z_{>1}$. These Ramanujan series can be reformulated as integrations of modular forms, in the framework of Eichler integrals. Particular cases of    these Eichler integrals recover part of the recent results reported by Wan and Zucker (\texttt{arXiv:1410.7081v1}). \textit{Subject Classification (AMS 2010)}:          11E45, 11M06, 33C75, 33E05\end{abstract}



\section{Introduction and Statement of Results}

Following Gross and Zagier \cite[][p.~239, Eq.~2.14]{GrossZagierI}, we define the Epstein zeta function on the Hecke congruence groups $ \varGamma_0(N):=\left\{ \left.\left(\begin{smallmatrix}a&b\\c&d\end{smallmatrix}\right)\right| a,b,c,d\in\mathbb Z; ad-bc=1;c\equiv0\pmod N\right\},N\in\mathbb Z_{>1}$ as \begin{align}
E^{\varGamma_0(N)}(z,s):=\sum_{\hat\gamma\in\left.\left(\begin{smallmatrix}*&*\\0&*\end{smallmatrix}\right)\right\backslash\varGamma_0(N)}[\I(\hat\gamma z )]^{s},\quad z\in\mathfrak H,\R s>1.
\end{align}Here, $ \mathfrak H=\{z\in\mathbb C|\I z>0\}$ is the upper half-plane.
 The Epstein zeta function appears in the  asymptotic expansion for the corresponding automorphic Green's function (see \cite[][p.~240, Eq.~2.19]{GrossZagierI} or \cite[][p.~39, Eq.~6.5]{Hejhal1983}) \begin{align}
G_{s}^{\mathfrak H/\overline{\varGamma}_0(N)}(z_{1},z_{2})=\frac{4\pi}{1-2s}\frac{E^{\varGamma_0(N)}(z_{2},s)}{(\I z_1)^{s-1}}+O(e^{-(2\pi-0^{+})\I z_1}),\quad \R s>1,\I z_1\to+\infty.\label{eq:G_HeckeN_asympt}
\end{align} Here, the automorphic Green's function $ G_{s}^{\mathfrak H/\overline{\varGamma}_0(N)}(z_{1},z_{2})$ of level $N\in\mathbb Z_{>0}$ and weight $ 2s$ is defined as (\cite[][p.~207]{GrossZagier1985}, \cite[][pp.~238--239]{GrossZagierI} and \cite[][p.~544]{GrossZagierII})\begin{align}
G_{s}^{\mathfrak H/\overline{\varGamma}_0(N)}(z_{1},z_{2}):=-\sum_{\substack{a,b,c,d\in\mathbb Z\\ N\mid c,ad-bc=1}}Q_{s-1}
\left( 1+\frac{\left\vert z_{1} -\frac{a z_2+b}{cz_{2}+d}\right\vert ^{2}}{2\I z_1\I\frac{a z_2+b}{cz_{2}+d}} \right),\quad z_1\notin\varGamma_0(N)z_2,
\end{align}where    $Q_\nu $ is the Legendre function of the second kind $ Q_\nu(t):=\int_0^{\infty}(t+\sqrt{t^2-1}\cosh u)^{-\nu-1}\D u,t>1,\R\nu>-1$.

The Epstein zeta function is also known as the real-analytic Eisenstein series. The namesake is best interpreted by the following identity \cite[][p.~207]{GrossZagier1985}:\begin{align}
E^{\varGamma_0(1)}(z,s)\equiv E^{SL(2,\mathbb Z)}(z,s)=\frac{1}{2\zeta(2s)}\sum_{\substack{m,n\in\mathbb Z\\m^2+n^2\neq0}}\frac{(\I z)^s}{|mz+n|^{2s}},
\end{align}which is reminiscent of the complex-analytic Eisenstein series $ E_k$ of weight $k$:\begin{align} E_k(z):=\frac{1}{2\zeta(k)}\sum_{\substack{m,n\in\mathbb Z\\m^2+n^2\neq0}}\frac{1}{(mz+n)^{k}},\quad k\in\mathbb Z_{>2}.\label{eq:E_k_defn}\end{align} Furthermore, $ E^{\varGamma_0(1)}(z,s)\equiv E^{SL(2,\mathbb Z)}(z,s)$ is the building block for all the Epstein zeta functions $E^{\varGamma_0(N)}(z,s) $ of higher levels $ N\in\mathbb Z_{>1}$, in that \cite[][p.~240, Eq.~2.16]{GrossZagierI}:\begin{align}
E^{\varGamma_0(N)}(z,s)=\frac{1}{N^{s}\prod_{p\mid N}(1-p^{-2s})}\sum_{d\mid N}\frac{\mu(d)}{d^s}E^{\varGamma_0(1)}\left( \frac{Nz}{d},s \right),\label{eq:Epstein_HeckeN_expn}
\end{align} where the product $ \prod_{p\mid N}$ (resp.~the sum $ \sum_{d|N}$) is taken over prime (resp.~positive) divisors of $N$, and $ \mu(\cdot)$ is the M\"obius function satisfying  $ \sum_{n=1}^\infty\mu(n)n^{-s}=1/\zeta(s)$ for $ \R s>1$.

For certain CM points $ z$ (where $[\mathbb Q(z):\mathbb Q]=2 $), the values for the Epstein zeta function $ E^{\varGamma_0(1)}(z,s)$ are known to be connected to  the $ L$-functions of non-trivial Dirichlet characters \cite{Berndt1992,Williams1999}. The simplest example is \cite[][Eq.~2.1.35]{AGF_PartI}\begin{align}
E^{\varGamma_0(1)}(i,s)=\frac{2\zeta(s)L(s,\chi_{-4})}{\zeta(2s)}=\frac{2\zeta(s)}{\zeta(2s)}\sum_{n=0}^\infty\frac{(-1)^n}{(2n+1)^s},\quad \R s>1.\label{eq:Epstein_SL2Z_i_s}
\end{align}Here, $ L(2,\chi_{-4})$ is equal to Catalan's constant  $ G=\sum_{n=0}^\infty(-1)^n(2n+1)^{-2}$.

Amongst our modest goals in this note is to construct rapidly convergent series representations for the Epstein zeta functions $E^{\varGamma_0(N)}(z,s),z\in\mathfrak H,s\in\mathbb Z_{>1}$, drawing inspirations from an entry   \cite[][p.~276]{RN2} in  Ramanujan's second notebook, which presented exponentially fast convergent series for special values of the Riemann zeta function at odd positive integers. Our formulae for Epstein zeta functions  are essentially non-holomorphic derivatives of the aforementioned  series expansion due to Ramanujan \cite[][p.~276]{RN2}, so we will refer to them as ``Ramanujan series''. In \S\S\ref{sec:Epstein_SL2Z}--\ref{sec:Epstein_Hecke4},  we prove the results stated in the  theorem below.\begin{theorem}[Ramanujan Series for Epstein Zeta Functions]\label{thm:Ramanujan_series}\begin{enumerate}[label=\emph{(\alph*)}, ref=(\alph*), widest=a]\item\label{itm:Epstein_SL2Z} For $ z\in\mathfrak H$ and $ m\in\mathbb Z_{>0}$, we have the following infinite series representation of the Epstein zeta function $ E^{\varGamma_0(1)}(z,m+1)$:\begin{align}
E^{\varGamma_0(1)}(z,m+1)={}&(\I z)^{m+1}+\frac{\sqrt{\pi}\Gamma(m+\frac{1}{2})}{\Gamma(m+1)}\frac{\zeta(2m+1)}{\zeta(2m+2)}\frac{1}{(\I z)^{m}}\notag\\{}&-\frac{\pi(\I z)^{m}}{(-2)^{m-1}\Gamma(m+1)\zeta(2m+2) }\left( \frac{\partial}{\partial\I z}\frac{1}{\I z} \right)^{m}\R\sum_{n=1}^\infty\frac{1}{n^{2m+1}}\frac{1}{e^{2n\pi\frac{z}{i}}-1}.\label{eq:Epstein_SL2Z_Ramanujan_series_repn}
\end{align}\item\label{itm:Epstein_Hecke4} For $ z\in\mathfrak H$ and $ m\in\mathbb Z_{>0}$, we have the following infinite series representation of the Epstein zeta function $ E^{\varGamma_0(4)}(-1/(4z),m+1)$:\begin{align}&
E^{\varGamma_0(4)}\left( -\frac{1}{4z},m+1 \right)\notag\\={}&\frac{2^{2m+1}-1}{2^{2m+2}-1}\frac{\sqrt{\pi}\Gamma(m+\frac{1}{2})}{2\Gamma(m+1)}\frac{\zeta(2m+1)}{\zeta(2m+2)}\frac{1}{(4\I z)^m}\notag\\{}&+\frac{\pi(\I z)^{m}}{(-2)^{m-1}(2^{2m+2}-1)\Gamma(m+1)\zeta(2m+2)}\left( \frac{\partial}{\partial\I z}\frac{1}{\I z} \right)^{m}\R\sum_{n=0}^\infty\frac{1}{(2n+1)^{2m+1}}\frac{1}{e^{(2n+1)\pi\frac{2z+1}{i}}+1}.\label{eq:Epstein_Hecke4_Ramanujan_series_repn}
\end{align}\end{enumerate}\end{theorem}

 We devote \S\ref{sec:period_int}    to integral representations of $ E^{\varGamma_0(4)}(z,s)$ where $ s\in\{2,3,4,5,7\}$. In the results summarized by the following theorem,  the complete elliptic integral  $ \mathbf K$ is defined by  \begin{align} \mathbf K(\sqrt{t}):=\int_0^{\pi/2}\frac{\D\theta}{\sqrt{1-t\sin^2\theta}},\quad t\in\mathbb C\smallsetminus[1,+\infty);\end{align} the modular lambda function \begin{align}\lambda(z):=\frac{16[\eta(z/2)]^8[\eta(2z)]^{16}}{[\eta(z)]^{24}},\quad z\in\mathfrak H\end{align} is defined via the Dedekind eta function \begin{align} \eta(z):=e^{\pi iz/12}\prod_{n=1}^\infty(1-e^{2\pi inz}),\quad z\in\mathfrak H.\end{align}\begin{theorem}[Some Special  Integral Representations for Epstein Zeta Functions]\label{thm:int_repn_Epstein_zeta} If $z\in\mathfrak H$ satisfies the following inequalities:\begin{align}
|\R(2z+1)|<1,\quad\left\vert 2z+\frac{1}{2} \right\vert>\frac{1}{2},\quad \left\vert 2z+\frac{3}{2} \right\vert>\frac{1}{2},\label{eq:lambda_range_Epstein_Hecke4}
\end{align} then we have the following integral representations of $ E^{\varGamma_0(4)}(-1/(4z),s),s\in\{2,3,4,5,7\}$:{\allowdisplaybreaks\begin{align}{}&
E^{\varGamma_0(4)}\left( -\frac{1}{4z},2 \right)-\frac{21 \zeta (3)}{8 \pi ^3\I z}\notag\\={}&\R\int_0^{\lambda(2z+1)}\frac{3[\mathbf K(\sqrt{t})]}{4\pi^{3}\I z}^2\left[ \frac{i\mathbf K(\sqrt{1-t})}{\mathbf K(\sqrt{t})} -2z-1\right]\left[ \frac{i\mathbf K(\sqrt{1-t})}{\mathbf K(\sqrt{t})} -2\overline{z}-1\right]\D t,\label{eq:K_int_repn2}\\&E^{\varGamma_0(4)}\left( -\frac{1}{4z},3 \right)-\frac{1395 \zeta (5)}{256 \pi ^5(\I z)^2}\notag\\={}&\R\int_0^{\lambda(2z+1)}\frac{15(2t-1)[\mathbf K(\sqrt{t})]^4}{32\pi^{5}(\I z)^2}\left[ \frac{i\mathbf K(\sqrt{1-t})}{\mathbf K(\sqrt{t})} -2z-1\right]^2\left[ \frac{i\mathbf K(\sqrt{1-t})}{\mathbf K(\sqrt{t})} -2\overline{z}-1\right]^{2}\D t,\label{eq:K_int_repn3}\\{}&E^{\varGamma_0(4)}\left( -\frac{1}{4z},4 \right)-\frac{200025 \zeta (7)}{17408 \pi ^7(\I z)^3}\notag\\={}&\R\int_0^{\lambda(2z+1)}\frac{70[2-17t(1-t)][\mathbf K(\sqrt{t})]^6}{1088\pi^{7}(\I z)^3}\left[ \frac{i\mathbf K(\sqrt{1-t})}{\mathbf K(\sqrt{t})} -2z-1\right]^3\left[ \frac{i\mathbf K(\sqrt{1-t})}{\mathbf K(\sqrt{t})} -2\overline{z}-1\right]^{3}\D t,\label{eq:K_int_repn4}\\{}&E^{\varGamma_0(4)}\left( -\frac{1}{4z},5 \right)-\frac{50703975 \zeta (9)}{2031616 \pi ^9(\I z)^4}\notag\\={}&\R\int_0^{\lambda(2z+1)}\frac{315(2t-1)[1-31t(1-t)][\mathbf K(\sqrt{t})]^8}{15872\pi^{9}(\I z)^4}\left[ \frac{i\mathbf K(\sqrt{1-t})}{\mathbf K(\sqrt{t})} -2z-1\right]^4\left[ \frac{i\mathbf K(\sqrt{1-t})}{\mathbf K(\sqrt{t})} -2\overline{z}-1\right]^{4}\D t,\label{eq:K_int_repn5}\\&E^{\varGamma_0(4)}\left( -\frac{1}{4z},7 \right)-\frac{11506129710075 \zeta (13)}{91620376576 \pi ^{13}(\I z)^6}\notag\\={}&\R\int_0^{\lambda(2z+1)}\frac{3003(2t-1)[1-512t(1-t)+5461t^{2}(1-t)^{2}][\mathbf K(\sqrt{t})]^{12}}{22368256\pi^{13}(\I z)^6}\times\notag\\&\times\left[ \frac{i\mathbf K(\sqrt{1-t})}{\mathbf K(\sqrt{t})} -2z-1\right]^6\left[ \frac{i\mathbf K(\sqrt{1-t})}{\mathbf K(\sqrt{t})} -2\overline{z}-1\right]^{6}\D t,\label{eq:K_int_repn7}
\end{align}}where all the integration paths are  straight-line segments joining the end points.  \end{theorem}

The formulae  in Theorem \ref{thm:int_repn_Epstein_zeta} for non-holomorphic Eisenstein series provide an alternative perspective to some  integral identities recently reported by Wan and Zucker \cite[][\S3]{WanZucker2014} for certain  $ L$-functions $ L(s,\chi)$  with unrestricted values of $s$, in the setting of holomorphic Eisenstein series. Our derivations in this note are based solely on symmetry considerations, without invoking the use of  Jacobi theta functions as in the work of Wan and Zucker.
\section{Ramanujan Series for $ E^{\varGamma_0(1)}(z,s)$\label{sec:Epstein_SL2Z}}We recall that the Epstein zeta function $ E^{\varGamma_0(1)}(z,s)$ has the following asymptotic expansion \cite[][p.~240, Eq.~2.17]{GrossZagierI}\begin{align}
E^{\varGamma_0(1)}(z,s)=(\I z)^s+\frac{\sqrt{\pi}\Gamma(s-\frac{1}{2})}{\Gamma(s)}\frac{\zeta(2s-1)}{\zeta(2s)}\frac{1}{(\I z)^{s-1}}+O(e^{-(2\pi-0^+)\I z}).\label{eq:Epstein_SL2Z_asympt}
\end{align}In the next lemma, we show that the aforementioned asymptotic behavior, together with a couple of analytic qualifications, uniquely characterize the Epstein zeta function. As in \cite[][Lemma 2.0.1]{AGF_PartI}, we denote the (negative semi-definite) Laplace operator on $ \mathfrak H$ as\begin{align}
\Delta_z^{\mathfrak H}=(\I z)^2\left[ \frac{\partial^2}{\partial(\R z)^2}+ \frac{\partial^2}{\partial(\I z)^2}\right].
\end{align}\begin{lemma}[Uniqueness of Epstein Zeta Function]\label{lm:EZF}Suppose that $\R s>1$. If a smooth function $ F:\mathfrak H\longrightarrow\mathbb R$ has the following properties: \begin{enumerate}[leftmargin=*,  label=\emph{(EZF\arabic*)},ref=(EZF\arabic*),
widest=a, align=left]\item \label{itm:EZF1}\emph{(\textbf{Symmetry})} $ F(z)=F(z+1)=F(-1/z),\forall z\in\mathfrak H$,\item \label{itm:EZF2} \emph{(\textbf{Differential Equation})} $ \Delta_z^{\mathfrak H}F(z)=s(s-1)F(z),\forall z\in\mathfrak H$,\item \label{itm:EZF3}\emph{(\textbf{Asymptotic Behavior})} $ F(iy)=y^s+o(y^s),\mathbb R\ni y\to+\infty$,\end{enumerate} then $ F(z)=E^{\varGamma_0(1)}(z,s)$.\end{lemma}\begin{proof}With the periodicity $ F(z)=F(z+1)$ in \ref{itm:EZF1} and the differential equation in \ref{itm:EZF2}, we have a Fourier expansion:\begin{align}
F(z)=A(\I z)^s+\frac{B}{(\I z)^{s-1}}+\sum_{n\in \mathbb Z\smallsetminus\{0\}}\sqrt{\I z}K_{s-\frac{1}{2}}(2\pi|n|\I z)e^{2\pi in\R z}c_n\label{eq:f_z_Fourier}
\end{align}for some constants $ A$, $B$ and $ \{c_n|n\in\mathbb Z,n\neq0\}$, where \cite[][\S6.15]{Watson1995Bessel}\begin{align} K_\nu(y)=\frac{\sqrt{\pi}(y/2)^\nu}{\Gamma(\nu+\frac{1}{2})}\int_0^\infty e^{-y\cosh u}(\sinh u)^{2\nu}\D u\end{align} is the  $K$-Bessel function. Once a smooth and real-valued function $ F(z),z\in\mathfrak H$ satisfies all the conditions in \ref{itm:EZF1}--\ref{itm:EZF3}, we see from Eq.~\ref{eq:f_z_Fourier} that $ F(z)-E^{\varGamma_0(1)}(z,s)$ defines a bounded function on the orbit space $ \varGamma_0(1)\backslash\mathfrak H$, and it  is annihilated by the operator  $ \Delta_z^{\mathfrak H}-s(s-1)$. Now that $ \R s>1$, the number $ s(s-1)$ cannot be an eigenvalue of the Laplacian $ \Delta_z^{\mathfrak H}$, so we must have a vanishing identity $ F(z)-E^{\varGamma_0(1)}(z,s)\equiv 0$.   \end{proof}

The formula we proposed in  Eq.~\ref{eq:Epstein_SL2Z_Ramanujan_series_repn}
is closely related to a well-known entry \cite[][p.~276]{RN2} from Ramanujan's second notebook, which we recapitulate in the lemma below.\begin{lemma}[Ramanujan's Reflection Formula]For  $ z\in\mathfrak H,m\in\mathbb Z_{>0}$, we have the following identity:\begin{align}&
\frac{1}{(z/i)^m}\sum_{n=1}^\infty\frac{1}{n^{2m+1}}\frac{1}{e^{2n\pi\frac{z}{i} }-1}-(-z/i)^{m}\sum_{n=1}^\infty\frac{1}{n^{2m+1}}\frac{1}{e^{2n\pi\frac{i}{z} }-1}\notag\\={}&\frac{\zeta(2m+2)}{2\pi}\left[\frac{1}{(z/i)^{m+1}}+(-z/i)^{m+1}\right]-\frac{\zeta(2m+1)}{2}\left[\frac{1}{(z/i)^{m}}-(-z/i)^{m}\right]+\frac{1}{\pi}\sum_{k=0}^{m-1}\frac{\zeta(2k+2)\zeta(2m-2k)}{(-1)^{k}(z/i)^{m-2k-1}}.
\label{eq:Ramanujan_reflection_notebook}\end{align}\end{lemma}\begin{proof}We follow the standard procedures in Grosswald's lemma \cite[][\S4]{GunMurtyRath2011}, starting from a Mellin inversion formula \cite[][p.~312, \S6.3, Eq.~7]{ET1}:\begin{align}
\frac{1}{e^{2n\pi y}-1}=\frac{1}{2\pi i}\int_{\frac{3}{2}-i\infty}^{\frac{3}{2}+i\infty}\frac{\Gamma(s)\zeta(s)\D s}{(2n\pi y)^{s}},\quad y>0.
\end{align}Clearly, we have \begin{align}\sum_{n=0}^
\infty\frac{1}{n^{2m+1}}\frac{1}{e^{2n\pi y}-1}=\frac{1}{2\pi i}\int_{\frac{3}{2}-i\infty}^{\frac{3}{2}+i\infty}\frac{R_m(s)\D s}{ y^{s}},\quad y>0,\label{eq:Rm_s}
\end{align} where the expression\begin{align}
R_m(s)=\frac{\Gamma(s)\zeta(s)\zeta(s+2m+1)}{(2\pi)^{s}}=\frac{\zeta(1-s)\zeta(s+2m+1)}{2\sin\frac{\pi(1-s)}{2}}
\end{align}satisfies a functional equation  $ R_{m}(s-2m)=R_m(-s)\cos\frac{\pi s}{2}\sec\frac{2m\pi-\pi s}{2}$. This reduces to a reflection formula $ R_m(s-2m)=(-1)^m R_{m}(-s)$ for $ m\in\mathbb Z_{>1}$.

Picking up  residues of $ R_m(s)$ at $ s=1,0$ and $ s=-2k-1$ for $ k\in\mathbb Z\cap[0,m-1]$, we have\begin{align}
\frac{1}{2\pi i}\int_{\frac{3}{2}-i\infty}^{\frac{3}{2}+i\infty}\frac{R_{m}(s)\D s}{ y^{s+m}}={}&\frac{1}{2\pi i}\int_{\frac{1}{2}-i\infty}^{\frac{1}{2}+i\infty}\frac{R_{m}(s-2m)\D s}{ y^{s-m}}+\frac{\zeta(2m+2)}{2\pi y^{m+1}}-\frac{\zeta(2m+1)}{2y^{m}}\notag\\{}&+\frac{1}{\pi}\sum_{k=0}^{m-1}\frac{\zeta(2k+2)\zeta(2m-2k)}{(-1)^{k}y^{m-2k-1}},\quad y>0,m\in\mathbb Z_{>1}.
\end{align}   However, by the reflection formula  $ R_m(s-2m)=(-1)^m R_{m}(-s)$ for $ m\in\mathbb Z_{>1}$, we obtain\begin{align}
\frac{1}{2\pi i}\int_{\frac{3}{2}-i\infty}^{\frac{3}{2}+i\infty}\frac{R_{m}(s-2m)\D s}{ y^{s-m}}={}&\frac{(-y)^m}{2\pi i}\int_{\frac{1}{2}-i\infty}^{\frac{1}{2}+i\infty}\frac{R_{m}(-s)\D s}{ y^{s}}=\frac{(-y)^m}{2\pi i}\int_{-\frac{1}{2}-i\infty}^{-\frac{1}{2}+i\infty}\frac{R_{m}(s)\D s}{ (1/y)^{s}}\notag\\={}&\frac{(-y)^m}{2\pi i}\int_{\frac{3}{2}-i\infty}^{\frac{3}{2}+i\infty}\frac{R_{m}(s)\D s}{ (1/y)^{s}}+(-y)^{m}\frac{\zeta(2m+1)}{2}-(-y)^{m}\frac{\zeta(2m+2)}{2\pi/y }.
\end{align}  This proves Eq.~\ref{eq:Ramanujan_reflection_notebook} for $ z/i>0$, and the rest follows from analytic continuation.
    \end{proof}With the preparations above, we can proceed with a proof of Eq.~\ref{eq:Epstein_SL2Z_Ramanujan_series_repn}.

\begin{proof}[Proof of Theorem~\ref{thm:Ramanujan_series}\ref{itm:Epstein_SL2Z}]We denote the real-valued function on the  right-hand side of  Eq.~\ref{eq:Epstein_SL2Z_Ramanujan_series_repn} by $ F(z)$.  The translational invariance $ F(z)=F(z+1)$ in \ref{itm:EZF1} is easy to check, while the relation $ F(z)=F(-1/z)$ in \ref{itm:EZF1} requires more efforts,  occupying three paragraphs to follow.

Firstly, we verify that $ F(iy)=F(i/y)$ for $y>0 $. Towards this end, we enlist the help of Eq.~\ref{eq:Rm_s} to compute\begin{align}-\frac{\pi}{(-2)^{m-1}\Gamma(m+1)\zeta(2m+2) }y^{m}\left( \frac{\partial}{\partial y}\frac{1}{y} \right)^m
\sum_{n=0}^
\infty\frac{1}{n^{2m+1}}\frac{1}{e^{2n\pi y}-1}=\frac{1}{2\pi i}\int_{\frac{3}{2}-i\infty}^{\frac{3}{2}+i\infty}\frac{\rho_m(s)\D s}{ y^{s+m}},
\end{align}where the function \begin{align}
\rho_m(s):=\frac{2\pi\Gamma \left(\frac{s+1}{2}+m\right)R_m(s)}{\Gamma \left(\frac{s+1}{2}\right)\Gamma(m+1)\zeta(2m+2)}=\frac{\Gamma\left( \frac{s}{2} \right)\zeta(s)\Gamma\left( \frac{s+1+2m}{2} \right)\zeta(s+2m+1)}{\pi^{s-\frac{1}{2}}\Gamma(m+1)\zeta(2m+2)}
\label{eq:rho_m_s_SL2Z}\end{align}satisfies a reflection formula $ \rho_m(s-2m)=\rho_m(-s)$ for arbitrary $m$. Unlike the function $ R_m(s)$, the only singularities for $ \rho_m(s)$ are two simple poles at   $ s=0,1$. This is because in the expression for $ \rho_m(s)$,  the simple poles of the Euler gamma function at negative integers are cancelled out by the trivial zeros of the Riemann zeta function at negative even numbers. By residue calculus, we can deduce {\allowdisplaybreaks\begin{align}
\frac{1}{2\pi i}\int_{\frac{3}{2}-i\infty}^{\frac{3}{2}+i\infty}\frac{\rho_m(s)\D s}{ y^{s+m}}={}&\frac{1}{2\pi i}\int_{\frac{3}{2}-i\infty}^{\frac{3}{2}+i\infty}\frac{\rho_m(s-2m)\D s}{ y^{s-m}}+\frac{1}{y^{m+1}}-\frac{\sqrt{\pi}\Gamma(m+\frac{1}{2})}{\Gamma(m+1)}\frac{\zeta(2m+1)}{\zeta(2m+2)}\frac{1}{y^m}\notag\\={}&\frac{1}{2\pi i}\int_{-\frac{3}{2}-i\infty}^{-\frac{3}{2}+i\infty}\frac{\rho_m(s)\D s}{ y^{-s-m}}+\frac{1}{y^{m+1}}-\frac{\sqrt{\pi}\Gamma(m+\frac{1}{2})}{\Gamma(m+1)}\frac{\zeta(2m+1)}{\zeta(2m+2)}\frac{1}{y^m}\notag\\={}&\frac{1}{2\pi i}\int_{\frac{3}{2}-i\infty}^{\frac{3}{2}+i\infty}\frac{\rho_m(s)\D s}{ (1/y)^{s+m}}+\left(\frac{1}{y^{m+1}}-y^{m+1}\right)\notag\\{}&-\frac{\sqrt{\pi}\Gamma(m+\frac{1}{2})}{\Gamma(m+1)}\frac{\zeta(2m+1)}{\zeta(2m+2)}\left(\frac{1}{y^m}-y^{m}\right),\quad \forall y>0,
\end{align}}which is the claimed symmetry $ F(iy)=F(i/y),\forall y>0$.

Secondly, we show that \begin{align}
\left.\frac{\partial^n}{\partial(\R z)^n}\right|_{\R z=0}\left[F(z)-F\left( -\frac{1}{z} \right)\right]=0,\quad \I z>0,n\in\mathbb Z_{>0}.\label{eq:high_order_deriv_vanish}
\end{align}
Since the function $ F(z)$ obviously satisfies $ F(z)=F(-\overline z)$, the equation above is evident when $n$ is a positive odd integer. For positive even integers $ n=2\ell\in2\mathbb Z_{>0}$, we can build\begin{align}
\left.\frac{\partial^{2\ell}}{\partial(\R z)^{2\ell}}\right|_{\R z=0}\left[F(z)-F\left( -\frac{1}{z} \right)\right]=0,\quad \I z>0\tag{\ref{eq:high_order_deriv_vanish}$^{(2\ell)}$}
\end{align}inductively  on the property \ref{itm:EZF2}. Concretely speaking,
as we have \begin{align}
[\Delta_z^{\mathfrak H}-m(m+1)]\left[ (\I z)^{m}\left( \frac{\partial}{\partial\I z}\frac{1}{\I z} \right)^{m}\R h(z) \right]=0
\end{align}for any holomorphic function $ h(z),z\in\mathfrak H$, we can confirm the differential equation  $ [\Delta_z^{\mathfrak H}-m(m+1)]F(z)=0,z\in\mathfrak H$. By the invariance property of the Laplace operator $ \Delta_z^{\mathfrak H}$, one can also show that  $ [\Delta_z^{\mathfrak H}-m(m+1)]F(-1/z)=0,z\in\mathfrak H$. For any positive integer $k$, we can decompose the left-hand side of the following equation:\begin{align}
[\Delta_z^{\mathfrak H}-m(m+1)]^k\left[F(z)-F\left( -\frac{1}{z} \right)\right]=0,\quad\R z=0 ,\I z>0,
\end{align}so as to show that  the truthfulness of Eq.~\ref{eq:high_order_deriv_vanish}$ ^{(2k)}$ hinges on  that of  Eq.~\ref{eq:high_order_deriv_vanish}$ ^{(2\ell)}$ for $ \ell\in\mathbb Z\cap[0,k)$. (We count   $ F(iy)=F(i/y),\forall y>0$ as the case of Eq.~\ref{eq:high_order_deriv_vanish}$ ^{(0)}$.)
This completes the verification of Eq.~\ref{eq:high_order_deriv_vanish} for all $n\in\mathbb Z_{>0}$.

Thirdly, we point out that the function $ F(z)-F(-1/z)$, which is annihilated by the differential operator $ \Delta_z^{\mathfrak H}-m(m+1)$, admits a convergent power series in an open neighborhood of $ z=i$. With the information input from the last two paragraphs, we see that  $ F(z)-F(-1/z)=0$ holds in a certain open neighborhood of the point $ z=i$. By the principle of unique continuation \cite[][p.~262]{BersJohnSchechter1964},
 the function $ F(z)-F(-1/z)$ must vanish identically for all  $ z\in\mathfrak H$.

Finally, the asymptotic behavior $ F(iy)=y^s+o(y^s)$ in \ref{itm:EZF3} is evident from the right-hand side of Eq.~\ref{eq:Epstein_SL2Z_Ramanujan_series_repn}, which concludes the proof. \end{proof}
\begin{remark}
Indeed, the Ramanujan series for $ E^{\varGamma_0(1)}(m+1)$ (Eq.~\ref{eq:Epstein_SL2Z_Ramanujan_series_repn}) is equivalent to the $ s=m+1$  case in the following standard Fourier expansion (see \cite[][p.~65, Proposition 8.6]{Hejhal1983} and \cite[][p.~207]{GrossZagier1985}):\begin{align}
E^{\varGamma_0(1)}(z,s)={}&(\I z)^s+\frac{\sqrt{\pi}\Gamma(s-\frac{1}{2})}{\Gamma(s)}\frac{\zeta(2s-1)}{\zeta(2s)}\frac{1}{(\I z)^{s-1}}\notag\\&+\frac{2\pi^{s}}{\Gamma(s)\zeta(2s)}\sum_{n\in\mathbb Z\smallsetminus\{0\}}|n|^{s-\frac{1}{2}}\sigma_{1-2s}(n)\sqrt{\I z}K_{s-\frac{1}{2}}(2\pi|n|\I z)e^{2\pi in\R z}\label{eq:Epstein_SL2Z_Fourier_Hejhal}
\end{align}  where $ \sigma_\nu(n)=\sum_{d|n}d^\nu$. The equivalence can be seen from a rearrangement of the series \cite[][p.~277, Entry 12.3.9]{RLN4}\begin{align}\sum_{n=1}^
\infty\frac{1}{n^{2m+1}}\frac{1}{e^{2n\pi y}-1}=\sum_{n=1}^
\infty\sum_{\ell=1}^\infty\frac{e^{-2\ell n\pi y}}{n^{2m+1}}=\sum_{n=1}^
\infty \sigma _{-2m-1}(n)e^{-2n\pi y}\label{eq:divisor_sum_rearrangement}
\end{align}and the recursion relation between contiguous $K$-Bessel functions.  It is perhaps worth noting that our demonstration of Eq.~\ref{eq:Epstein_SL2Z_Ramanujan_series_repn} exploits only symmetry, and involves no explicit computations for the Fourier coefficients using Hejhal's double-coset decomposition \cite[][p.~65, Proposition 8.6]{Hejhal1983}. It is not hard to extend our method in proving Eq.~\ref{eq:Epstein_SL2Z_Fourier_Hejhal}  for $ s=m+1\in\mathbb Z_{>1}$ to generic $s$. We omit the details.   \eor\end{remark}\begin{remark}Special cases of  Eq.~\ref{eq:Epstein_SL2Z_Ramanujan_series_repn} lead to some interesting evaluations of infinite series. The simplest example among them might be\begin{align}
\sum_{n=1}^\infty\frac{1}{n^2\sinh^2(n\pi)}=\frac{2G}{3}-\frac{11\pi^2}{180}.\label{eq:G_sinh}
\end{align} To prove Eq.~\ref{eq:G_sinh},  one uses  Eq.~\ref{eq:Epstein_SL2Z_Ramanujan_series_repn} to spell out \begin{align} \frac{30G}{\pi^{2}}=E^{\varGamma_0(1)}(i,2)=1+\frac{45\zeta(3)}{\pi^{3}}+\frac{90}{\pi^3}\sum_{n=1}^\infty\frac{1}{n^3}\frac{1}{e^{2n\pi}-1}+\frac{45}{\pi^2}\sum_{n=1}^\infty\frac{1}{n^2\sinh^2(n\pi)},\end{align} and eliminates from the equation above the following identity: \begin{align}
\sum_{n=1}^\infty\frac{1}{n^3}\frac{1}{e^{2n\pi}-1}=\frac{7\pi^3}{360}-\frac{\zeta(3)}{2},
\end{align}which arises from a special case ($m=1, z=i$)  of Eq.~\ref{eq:Ramanujan_reflection_notebook}.  \eor\end{remark}
\section{Ramanujan Series for $ E^{\varGamma_0(4)}(z,s)$\label{sec:Epstein_Hecke4}}From  Eq.~\ref{eq:Epstein_HeckeN_expn}, we know that $ E^{\varGamma_0(1)}(\cdot,s)$ determines all the Epstein zeta functions $ E^{\varGamma_0(N)}(\cdot,s)$ on Hecke congruence groups $ \varGamma_0(N),N\in\mathbb Z_{>0}$. In the next lemma, we show that the function $ E^{\varGamma_0(4)}(\cdot,s)$,  $\R s>1$ also encodes the complete information for all the Epstein zeta functions   $ E^{\varGamma_0(N)}(\cdot,s)$, $N\in\mathbb Z_{>0}$,  $\R s>1$.
\begin{lemma}[Some Addition Formulae for Epstein Zeta Functions]For $ \R s>1$ and $ z\in\mathfrak H $, we have the following algebraic relations among Epstein zeta functions:{\allowdisplaybreaks\begin{align}
E^{\varGamma_0(4)}(z,s)={}&E^{\varGamma_0(4)}\left(z+\frac{1}{2},s\right),\label{eq:Epstein_Hecke4_shift}\\E^{\varGamma_0(2)}(z,s)={}&2^sE^{\varGamma_0(4)}\left( \frac{z}{2} ,s\right)=E^{\varGamma_0(4)}(z,s)+E^{\varGamma_0(4)}\left( -\frac{1}{2(2z+1)} ,s\right),\label{eq:Epstein_Hecke2_add_form}\\E^{\varGamma_0(1)}(z,s)={}&2^{s}\left[ E^{\varGamma_0(4)}\left( \frac{z}{2} ,s\right)+2^{s}E^{\varGamma_0(4)}\left( -\frac{1}{4z} ,s\right)  \right],\label{eq:Epstein_SL2Z_add_form}\\E^{\varGamma_0(2)}(z,s)={}&\frac{1}{2^s-2^{-s}}\left[ E^{\varGamma_0(1)}(2z,s)-\frac{E^{\varGamma_0(1)}(z,s)}{2^{s}} \right],\label{eq:Epstein_Hecke2_expn}\\E^{\varGamma_0(4)}(z,s)={}&\frac{2^{-s}}{2^s-2^{-s}}\left[ E^{\varGamma_0(1)}(4z,s)-\frac{E^{\varGamma_0(1)}(2z,s)}{2^{s}} \right].\label{eq:Epstein_Hecke4_expn}
\end{align}}We accordingly have the following asymptotic expansions near the cusps of $ \varGamma_0(4)\backslash\mathfrak H^*$:\begin{align}
E^{\varGamma_0(4)}(z,s)={}&(\I z)^s+O\left( \frac{1}{(\I z)^{s-1}} \right),&& z\to i\infty,\\E^{\varGamma_0(4)}(z,s)={}&\frac{2^{s-1}-2^{-s}}{2^s-2^{-s}}\frac{\sqrt{\pi}\Gamma(s-\frac{1}{2})}{2\Gamma(s)}\frac{\zeta(2s-1)}{\zeta(2s)}(\I z)^{s-1}+O(e^{-\frac{2\pi-0^{+}}{4\I z}}),&& z\to i0^+ \emph{ or }\frac{1}{2}+i0^+.\label{eq:Hecke4_bottom_cusp_expn}
\end{align}\end{lemma}\begin{proof}One can verify the relation \begin{align}G_s^{\mathfrak H/\overline{\varGamma}_0(4)}(z,z')=G_s^{\mathfrak H/\overline{\varGamma}_0(4)}\left(z+\frac{1}{2},z'+\frac{1}{2}\right)\label{eq:Hecke4_double_shift}\end{align}  from the fact that $\left( \begin{smallmatrix}1&1/2\\0&1\end{smallmatrix} \right):z\mapsto z+\frac{1}{2}$ normalizes the Hecke congruence group $ \varGamma_0(4)$:\begin{align}\begin{pmatrix}1 & \frac{1}{2} \\
0 & 1 \\
\end{pmatrix}\begin{pmatrix}a & b \\
4c & d \\
\end{pmatrix}=\begin{pmatrix}a+2c & b-c+\frac{d-a}{2} \\
4c & d-2c \\
\end{pmatrix}\begin{pmatrix}1 & \frac{1}{2} \\
0 & 1 \\
\end{pmatrix},\quad\text{where }a,b,c,d\in\mathbb Z,ad-4bc=1.\label{eq:Hecke4_double_shift_reason}\end{align}
Thus,
Eq.~\ref{eq:Epstein_Hecke4_shift} follows from asymptotic analysis on Eq.~\ref{eq:Hecke4_double_shift}.

From the following addition formula for automorphic Green's functions \cite[][Eq.~2.2.7]{AGF_PartI}\begin{align}
G_s^{\mathfrak H/\overline{\varGamma}_0(2)}(z,z')={}&G_s^{\mathfrak H/\overline{\varGamma}_0(4)}\left( \frac{z}{2} ,\frac{z'}{2}\right)+G_s^{\mathfrak H/\overline{\varGamma}_0(4)}\left( \frac{z+1}{2} ,\frac{z'}{2}\right),
\end{align} one can deduce the corresponding addition formula for Epstein zeta functions:\begin{align}
E^{\varGamma_0(2)}(z,s)={}&\frac{1-2s}{4\pi}\lim_{\I z'\to+\infty}(\I z')^{s-1}G_2^{\mathfrak H/\overline{\varGamma}_0(2)}(z,z')\notag\\={}&\frac{1-2s}{4\pi}2^{s-1}\lim_{\I z'\to+\infty}\left( \I\frac{z'}{2} \right)^{s-1}\left[ G_2^{\mathfrak H/\overline{\varGamma}_0(4)}\left( \frac{z}{2} ,\frac{z'}{2}\right)+G_2^{\mathfrak H/\overline{\varGamma}_0(4)}\left( \frac{z+1}{2} ,\frac{z'}{2}\right) \right]\notag\\={}&2^{s-1}\left[ E^{\varGamma_0(4)}\left( \frac{z}{2},s \right)+E^{\varGamma_0(4)}\left( \frac{z+1}{2},s \right)\right]=2^sE^{\varGamma_0(4)}\left( \frac{z}{2},s \right),
\end{align}which forms the first equality in Eq.~\ref{eq:Epstein_Hecke2_add_form}.  Similarly, the limit behavior of another addition formula  \cite[][Eq.~2.2.8]{AGF_PartI}\begin{align}
G_s^{\mathfrak H/\overline{\varGamma}_0(2)}(z,z')={}&G_s^{\mathfrak H/\overline{\varGamma}_0(4)}\left( z+\frac{1}{2} ,z'+\frac{1}{2}\right)+G_s^{\mathfrak H/\overline{\varGamma}_0(4)}\left( -\frac{1}{2(2z+1)} ,z'+\frac{1}{2}\right)
\end{align}brings us the second equality in Eq.~\ref{eq:Epstein_Hecke2_add_form}.

We can reformulate \cite[][Eq.~2.2.6]{AGF_PartI}  into the following form:\begin{align}&
G_s^{\mathfrak H/PSL(2,\mathbb Z)}(z,z')\notag\\={}&\left[ G_s^{\mathfrak H/\overline{\varGamma}_0(4)}\left( \frac{z}{2} ,\frac{z'}{2}\right)+G_s^{\mathfrak H/\overline{\varGamma}_0(4)}\left( \frac{z+1}{2} ,\frac{z'}{2}\right) \right]+\left[ G_s^{\mathfrak H/\overline{\varGamma}_0(4)}\left( -\frac{1}{2z} ,\frac{z'}{2}\right)+G_s^{\mathfrak H/\overline{\varGamma}_0(4)}\left( -\frac{1}{2z} +\frac{1}{2},\frac{z'}{2}\right) \right]\notag\\& +\left[G_s^{\mathfrak H/\overline{\varGamma}_0(4)}\left(-\frac{1}{2(-\frac{1}{z}+1)},\frac{z'}{2}\right)+G_s^{\mathfrak H/\overline{\varGamma}_0(4)}\left(-\frac{1}{2(-\frac{1}{z}+1)}+\frac{1}{2},\frac{z'}{2}\right)\right],
\end{align}which leads us to\begin{align}
2^{-s}E^{\varGamma_0(1)}(z,s)=E^{\varGamma_0(4)}\left( \frac{z}{2} ,s\right)+E^{\varGamma_0(4)}\left( -\frac{1}{2z} ,s\right) +E^{\varGamma_0(4)}\left( -\frac{1}{2(-\frac{1}{z}+1)} ,s\right).
\end{align}As we combine the last two addends using  Eq.~\ref{eq:Epstein_Hecke2_add_form}, we see that Eq.~\ref{eq:Epstein_SL2Z_add_form} is true.

The identities in Eqs.~\ref{eq:Epstein_Hecke2_expn} and \ref{eq:Epstein_Hecke4_expn} follow immediately from Eq.~\ref{eq:Epstein_HeckeN_expn}.
 Comparing Eqs.~\ref{eq:Epstein_Hecke2_expn} and \ref{eq:Epstein_Hecke4_expn}, we also recover the first equality in  Eq.~\ref{eq:Epstein_Hecke2_add_form}.

As we have the limit behavior of $E^{\varGamma_0(1)}(z,s)$  (Eq.~\ref{eq:Epstein_SL2Z_asympt}) and $  \sum_{d\mid N}\mu(d)d^{-2s}=\prod_{p\mid N}(1-p^{-2s}) $, the asymptotic expansion $E^{\varGamma_0(N)}(z,s)=(\I z)^s+O((\I z)^{1-s})$,  $z\to i\infty $ is true for all $ N\in\mathbb Z_{>0}$ \cite[][p.~240]{GrossZagierI}.  As $ z\to i0^+$, we can use the $ SL(2,\mathbb Z)$-invariance of $ E^{\varGamma_0(1)}(z,s)$ to argue that \begin{align}
E^{\varGamma_0(4)}(z,s)={}&\frac{2^{-s}}{2^s-2^{-s}}\left[ E^{\varGamma_0(1)}\left( -\frac{1}{4z},s \right)-\frac{1}{2^{s}} E^{\varGamma_0(1)}\left( -\frac{1}{2z},s \right)\right]\notag\\={}&\frac{2^{-s}}{2^s-2^{-s}}\frac{\sqrt{\pi}\Gamma(s-\frac{1}{2})}{\Gamma(s)}\frac{\zeta(2s-1)}{\zeta(2s)}\left\{\frac{1}{[\I (-\frac{1}{4z})]^{s-1}}-\frac{1}{2^{s}}\frac{1}{[\I (-\frac{1}{2z})]^{s-1}}\right\}+O(e^{-\frac{2\pi-0^{+}}{4\I z}}),
\end{align}thereby confirming Eq.~\ref{eq:Hecke4_bottom_cusp_expn}.\end{proof}\begin{lemma}[A Reflection Formula of Ramanujan Type]For $ z\in\mathfrak H,m\in\mathbb Z_{>0}$, we have the following identity:\begin{align}&
\frac{1}{(2z/i)^m}\sum_{n=0}^\infty\frac{1}{(2n+1)^{2m+1}}\frac{1}{e^{(2n+1)\pi\frac{2z}{i} }+1}-(-2z/i)^{m}\sum_{n=0}^\infty\frac{1}{(2n+1)^{2m+1}}\frac{1}{e^{(2n+1)\pi\frac{i}{2z} }+1}\notag\\={}&\frac{1-2^{-2m-1}}{2}\zeta(2m+1)\left[\frac{1}{(2z/i)^{m}}-(-2z/i)^{m}\right]+\frac{2}{\pi}\sum_{k=0}^{m-1}\frac{\zeta(2k+2)\zeta(2m-2k)(2^{-2k-2}-1)(1-2^{2k-2m})}{(-1)^{k}(2z/i)^{m-2k-1}}.\label{eq:Ramanujan_reflection}
\end{align}\end{lemma}\begin{proof}We   use the Mellin transform and contour deformation, as in the proof of Eq.~\ref{eq:Ramanujan_reflection_notebook}.

Without loss of generality, we momentarily restrict our analysis to the case where $ z/i>0$, and use the following Mellin inversion formula \cite[][p.~312, \S6.3, Eq.~6]{ET1}:\begin{align}
\frac{1}{e^{(2n+1)\pi y}+1}=\frac{1}{2\pi i}\int_{\frac{1}{2}-i\infty}^{\frac{1}{2}+i\infty}\frac{\Gamma(s)\zeta(s)(1-2^{1-s})\D s}{[(2n+1)\pi y]^{s}},\quad y>0,
\end{align}which results in\begin{align}&
\frac{1}{y^m}\sum_{n=0}^\infty\frac{1}{(2n+1)^{2m+1}}\frac{1}{e^{(2n+1)\pi y}+1}=\frac{1}{2\pi i}\int_{\frac{1}{2}-i\infty}^{\frac{1}{2}+i\infty}\frac{r_m(s)\D s}{ y^{s+m}},\quad y>0,\notag\\&\text{where }r_m(s)=\frac{\Gamma(s)\zeta(s)\zeta(s+2m+1)(1-2^{1-s})(1-2^{-1-2m-s})}{\pi^{s}}.
\end{align} We note that the expression  $ r_m(s)=\zeta(1-s)\zeta(s+2m+1)(2^{s-1}-1)(1-2^{-1-2m-s})/\sin\frac{\pi(1-s)}{2}$ satisfies a  functional equation $ r_m(-s)=[\cos(m\pi)+\sin(m\pi)\tan\frac{\pi s}{2}]r_m(s-2m)$, which  reduces to a reflection formula $ r_m(-s)=(-1)^mr_m(s-2m)$ for  $m\in\mathbb Z_{>0}$. By residue calculus, we have\begin{align}
\frac{1}{2\pi i}\int_{\frac{1}{2}-i\infty}^{\frac{1}{2}+i\infty}\frac{r_m(s)\D s}{ y^{s+m}}={}&\frac{1}{2\pi i}\int_{\frac{1}{2}-i\infty}^{\frac{1}{2}+i\infty}\frac{r_m(s-2m)\D s}{ y^{s-m}}+\frac{1-2^{-2m-1}}{2}\frac{\zeta(2m+1)}{y^{m}}\notag\\{}&+\frac{2}{\pi}\sum_{k=0}^{m-1}\frac{\zeta(2k+2)\zeta(2m-2k)(2^{-2k-2}-1)(1-2^{2k-2m})}{(-1)^{k}y^{m-2k-1}}
\end{align} and\begin{align}
\frac{1}{2\pi i}\int_{\frac{1}{2}-i\infty}^{\frac{1}{2}+i\infty}\frac{r_m(s-2m)\D s}{ y^{s-m}}={}&\frac{(-y)^m}{2\pi i}\int_{\frac{1}{2}-i\infty}^{\frac{1}{2}+i\infty}\frac{r_m(-s)\D s}{ y^{s}}=\frac{(-y)^m}{2\pi i}\int_{-\frac{1}{2}-i\infty}^{-\frac{1}{2}+i\infty}\frac{r_m(s)\D s}{ (1/y)^{s}}\notag\\={}&\frac{(-y)^m}{2\pi i}\int_{\frac{1}{2}-i\infty}^{\frac{1}{2}+i\infty}\frac{r_m(s)\D s}{ (1/y)^{s}}-(-y)^{m}\frac{1-2^{-2m-1}}{2}\zeta(2m+1).
\end{align}This proves Eq.~\ref{eq:Ramanujan_reflection} for $ z/i>0$, and the generic case hinges on analytic continuation.
 \end{proof}\begin{remark}It is worth pointing out that, for $ z/i>0$,  the identity in Eq.~\ref{eq:Ramanujan_reflection} had been previously proved by Berndt \cite[][Theorem 4.7]{Berndt1978}, as a special case of his modular transformation formula for generalized Eisenstein series \cite[][Theorem 4.6]{Berndt1978}. I thank an anonymous referee for bringing Berdnt's work \cite{Berndt1978} to my attention. \eor\end{remark}
 
 We can now move on to the justification of Eq.~\ref{eq:Epstein_Hecke4_Ramanujan_series_repn}.\begin{proof}[Proof of Theorem~\ref{thm:Ramanujan_series}\ref{itm:Epstein_Hecke4}]First, we point out  a variation on Lemma~\ref{lm:EZF}, which uniquely characterizes the function $ \varPhi(z)=E^{\varGamma_0(4)}(z,s)$ as a smooth mapping from $ \mathfrak H$ to $ \mathbb R$ satisfying the properties  below:\begin{enumerate}[leftmargin=*,  label={(EZF\arabic*$'$)},
widest=a, align=left]\item \label{itm:EZF1'}{(\textbf{Symmetry})} $ \varPhi(z)=\varPhi((2z+1)/2)=\varPhi(-(2z+1)/(8z+2)),\forall z\in\mathfrak H$,\item \label{itm:EZF2'}{(\textbf{Differential Equation})} $ \Delta_z^\mathfrak H\varPhi(z)=s(s-1)\varPhi(z),\forall z\in\mathfrak H$,\item \label{itm:EZF3'}{(\textbf{Asymptotic Behavior})} $ \varPhi(iy)=y^{s}+o(y^s),\mathbb R\ni  y\to+\infty$; $ \varPhi(iy)=O(1),\mathbb R\ni  y\to0^+$.\end{enumerate}Here, in writing the conditions in \ref{itm:EZF1'}, we have effectively tested $ \varGamma_0(2)$-invariance of the function $ E^{\varGamma_0(4)}(z,s)=2^{-s}E^{\varGamma_0(2)}(2z,s)$ (Eq.~\ref{eq:Epstein_Hecke2_add_form}) on the two generators   $ \hat{T}=\left(\begin{smallmatrix}1&1\\0&1\end{smallmatrix}\right)$ and $ \hat{V}_1=\left(\begin{smallmatrix}1&1\\-2&-1\end{smallmatrix}\right)$~\cite[][Theorem~4.3]{ApostolVol2} for the projective Hecke congruence group $\overline{\varGamma}_0(2) $.

As we abbreviate our proposed identity  Eq.~\ref{eq:Epstein_Hecke4_Ramanujan_series_repn} into the form $ E^{\varGamma_0(4)}(-1/(4z),m+1)=\varPsi(2z+1)$, we see that  \ref{itm:EZF1'} is equivalent to the condition $ \varPsi(2z+1)=\varPsi(2z-1)=\varPsi(-1/(2z+1))$. Meanwhile, the condition  \ref{itm:EZF2'} is obviously satisfied by our proposed formula, and  the right-hand side of   Eq.~\ref{eq:Epstein_Hecke4_Ramanujan_series_repn}   goes to zero as $ z\to i\infty$. Thus, it remains to check    $\varPsi(2iy+1)=(4y)^{-m-1}+o(y^{-m-1}),\mathbb R\ni y\to0^+ $ before we can verify     \ref{itm:EZF3'}.

In this paragraph, we check   \ref{itm:EZF1'} in the form of  $ \varPsi(2z+1)=\varPsi(2z-1)=\varPsi(-1/(2z+1))$. The condition $ \varPsi(2z+1)=\varPsi(2z-1)$ follows directly from the periodicity of the exponential function. To establish the inversion symmetry   $ \varPsi(2z+1)=\varPsi(-1/(2z+1))$, it would suffice to verify $ \varPsi(iy)=\varPsi(i/y),\forall y>0$ (cf.~the proof of Eq.~\ref{eq:Epstein_SL2Z_Ramanujan_series_repn}). When $ 2z+1=iy$ with $ y>0$, we can deduce
\begin{align}&
\frac{\pi(\I z)^{m}}{(-2)^{m-1}(2^{2m+2}-1)\Gamma(m+1)\zeta(2m+2)}\left( \frac{\partial}{\partial\I z}\frac{1}{\I z} \right)^{m}\R\sum_{n=0}^\infty\frac{1}{(2n+1)^{2m+1}}\frac{1}{e^{(2n+1)\pi\frac{2z+1}{i}}+1}\notag\\={}&\frac{2\pi y^{m}}{(-1)^{m-1}(2^{2m+2}-1)\Gamma(m+1)\zeta(2m+2)}\left( \frac{\partial}{\partial y}\frac{1}{y} \right)^{m}\sum_{n=0}^\infty\frac{1}{(2n+1)^{2m+1}}\frac{1}{e^{(2n+1)\pi y}+1}\notag\\={}&\frac{1}{2\pi i}\int_{\frac{1}{2}-i\infty}^{\frac{1}{2}+i\infty}\frac{\varrho_m(s)\D s}{ y^{s+m}},
\end{align}where\begin{align}
\varrho_m(s):={}&-\frac{2^{m+1}\pi\Gamma \left(\frac{s+1}{2}+m\right)r_m(s)}{(2^{2m+2}-1)\Gamma \left(\frac{s+1}{2}\right)\Gamma(m+1)\zeta(2m+2)}\notag\\={}&-\frac{2^{m+1}\pi}{2^{2 m+2}-1 }\frac{(2^{s-1}-1)(1-2^{-1-2m-s})}{\pi ^{s+\frac{1}{2}} }\frac{  \Gamma \left(\frac{s}{2}\right) \zeta (s) \Gamma \left(\frac{s+1}{2}+m\right)\zeta (s+2m+1)}{ \Gamma (m+1)\zeta (2 m+2)}
\end{align}satisfies a reflection formula $ \varrho_m(s-2m)=\varrho_m(-s)$ for whatever $m\in\mathbb C$.
The only singularity of $ \varrho_m(s)$ is a simple pole at $ s=0$. Thus, by contour deformation and residue calculus, we arrive at\begin{align}
\frac{1}{2\pi i}\int_{\frac{1}{2}-i\infty}^{\frac{1}{2}+i\infty}\frac{\varrho_m(s)\D s}{ y^{s+m}}={}&\frac{1}{2\pi i}\int_{\frac{1}{2}-i\infty}^{\frac{1}{2}+i\infty}\frac{\varrho_m(s-2m)\D s}{ y^{s-m}}-\frac{2^{2m+1}-1}{2^{2m+2}-1}\frac{\sqrt{\pi}\Gamma(m+\frac{1}{2})}{2\Gamma(m+1)}\frac{\zeta(2m+1)}{\zeta(2m+2)}\frac{1}{(2y)^m}\notag\\={}&\frac{1}{2\pi i}\int_{-\frac{1}{2}-i\infty}^{-\frac{1}{2}+i\infty}\frac{\varrho_m(s)\D s}{ y^{-s-m}}-\frac{2^{2m+1}-1}{2^{2m+2}-1}\frac{\sqrt{\pi}\Gamma(m+\frac{1}{2})}{2\Gamma(m+1)}\frac{\zeta(2m+1)}{\zeta(2m+2)}\frac{1}{(2y)^m}\notag\\={}&\frac{1}{2\pi i}\int_{\frac{1}{2}-i\infty}^{\frac{1}{2}+i\infty}\frac{\varrho_m(s)\D s}{ (1/y)^{s+m}}-\frac{2^{2m+1}-1}{2^{2m+2}-1}\frac{\sqrt{\pi}\Gamma(m+\frac{1}{2})}{2^{m+1}\Gamma(m+1)}\frac{\zeta(2m+1)}{\zeta(2m+2)}\left( \frac{1}{y^m}-y^m \right),
\end{align}for all $ y>0$.
Therefore, we have  $ \varPsi(iy)=\varPsi(i/y),\forall y>0$, as claimed.

We now wrap up our proof with the confirmation of     $\varPsi(2iy+1)=(4y)^{-m-1}+o(y^{-m-1}),\mathbb R\ni y\to0^+ $. We may compute{\allowdisplaybreaks\begin{align}
\varPsi(2iy+1)={}&\frac{2^{2m+1}-1}{2^{2m+2}-1}\frac{\sqrt{\pi}\Gamma(m+\frac{1}{2})}{2\Gamma(m+1)}\frac{\zeta(2m+1)}{\zeta(2m+2)}\frac{1}{(4y)^m}\notag\\{}&-\frac{\pi y^{m}}{(-2)^{m-1}(2^{2m+2}-1)\Gamma(m+1)\zeta(2m+2)}\left( \frac{\partial}{\partial y}\frac{1}{y} \right)^{m}\R\sum_{n=0}^\infty\frac{1}{(2n+1)^{2m+1}}\frac{1}{e^{2(2n+1)\pi y}-1}\notag\\={}&\frac{2^{2m+1}-1}{2^{2m+2}-1}\frac{\sqrt{\pi}\Gamma(m+\frac{1}{2})}{2\Gamma(m+1)}\frac{\zeta(2m+1)}{\zeta(2m+2)}\frac{1}{(4y)^m}+\frac{1}{2\pi i}\int_{\frac{3}{2}-i\infty}^{\frac{3}{2}+i\infty}\frac{(1-2^{-1-2m-s})\rho_m(s)\D s}{ (2^{2m+2}-1)y^{s+m}}\notag\\={}&\frac{1}{(4y)^{m+1}}+\frac{1}{2\pi i}\int_{-\frac{1}{2}-i\infty}^{-\frac{1}{2}+i\infty}\frac{(1-2^{-1-2m-s})\rho_m(s)\D s}{ (2^{2m+2}-1)y^{s+m}} ,\label{eq:Epstein_Hecke4_asympt_contour}
\end{align}}where $ \rho_m(s)$ was defined in Eq.~\ref{eq:rho_m_s_SL2Z}. Here, in the last step of Eq.~\ref{eq:Epstein_Hecke4_asympt_contour}, we  have collected residues at $ s=1$ and $ s=0$. The remaining integral over the vertical line $ \R s=-1/2$ clearly contributes $ o(y^{-m-1})$ to $ \varPsi(iy+1)$, as desired. This completes our verification of the qualifications in \ref{itm:EZF1'}--\ref{itm:EZF3'}, so the right-hand side of  Eq.~\ref{eq:Epstein_Hecke4_Ramanujan_series_repn} is indeed a valid representation of $ E^{\varGamma_0(4)}(-1/(4z),m+1)$ for any  $ z\in\mathfrak H$ and $m\in\mathbb Z_{>0}$.
   \end{proof}\begin{remark}An analog of Eq.~\ref{eq:G_sinh} is the following evaluation:\begin{align}
\sum_{n=0}^\infty\frac{1}{(2n+1)^2\cosh^2\frac{(2n+1)\pi}{2}}=\frac{\pi^{2}}{16}-\frac{G}{2}.\label{eq:G_sosh}
\end{align}Such an identity originates from a special case of Eq.~\ref{eq:Epstein_Hecke2_expn}:\begin{align}
E^{\varGamma_0(2)}\left( \frac{1+i}{2},s \right)={}&\frac{1}{2^s-2^{-s}}\left[ E^{\varGamma_0(1)}(1+i,s)-\frac{E^{\varGamma_0(1)}\left( \frac{1}{1-i},s \right)}{2^{s}} \right]=\frac{E^{\varGamma_0(1)}(i,s)}{2^{s}+1}.\label{eq:Epstein_Hecke2_SL2Z_spec_pt}
\end{align}We leave the rest of the details (series expansion of $ E^{\varGamma_0(2)}((1+i)/2,2)$, back reference to Eq.~\ref{eq:Ramanujan_reflection}, etc.) to our readers.\eor\end{remark}

\section{Some Integral Formulations  for Ramanujan Series\label{sec:period_int}}Before we proceed, let us first recollect some facts about the modular lambda function $ \lambda(z),z\in\mathfrak H$  and the complete elliptic integral $\mathbf K(\sqrt{t}),t\in\mathbb C\smallsetminus[1,+\infty) $.

We define the  $ \varLambda$-group as $ \varLambda:=\left\{\left.\left(\begin{smallmatrix}2a+1&2b\\2c&2d+1\end{smallmatrix}\right)\right|a,b,c,d\in\mathbb Z,(2a+1)(2d+1)-4bc=1\right\}$. It characterizes the symmetry of the modular lambda function: we have $  \lambda(\hat\gamma z):=\lambda\left(\frac{az+b}{cz+d}\right)=\lambda(z)$ for any $ \hat \gamma=\left(\begin{smallmatrix}a&b\\ c&d\end{smallmatrix}\right)\in\varLambda$. Let \begin{align}
\Int\mathfrak D_\varLambda={}&\left\{ z\in\mathfrak H\left| |\R z|<1,\left|z+\frac{1}{2}\right|>\frac{1}{2},\left|z-\frac{1}{2}\right|>\frac{1}{2} \right. \right\}\label{eq:Int_D_Lambda}\end{align}be the interior of the fundamental domain for  the $ \varLambda$-group, then $ z\mapsto \lambda(z)$ induces a bijective map from $  \Int\mathfrak D_\varLambda$ to $ (\mathbb C\smallsetminus\mathbb R)\cup(0,1)$,  and we have  \begin{align}z={}&\frac{i\mathbf K(\sqrt{1-\lambda(z)})}{\mathbf K(\sqrt{\lambda(z)})},&z\in\Int\mathfrak D_\varLambda.\label{eq:lambda_K_rln}\end{align}The equation above entails the following relation \begin{align}
\lambda\left(-\dfrac{1}{2z+1}\right)&=1-\lambda(2z+1)\label{eq:lambda_inv}
\end{align}when $ 2z+1\in\Int\mathfrak D_\varLambda$ (see Eq.~\ref{eq:lambda_range_Epstein_Hecke4}).
As we combine the ``$ \lambda$-$ \mathbf K$ relation'' (Eq.~\ref{eq:lambda_K_rln}) with Landen's transformations\begin{align}
\mathbf K(\sqrt{1-\lambda})={}&\frac{2}{1+\sqrt{\lambda}}\mathbf K\left( \frac{1-\sqrt{\lambda}}{1+\sqrt{\lambda}} \right),\quad \lambda\in\mathbb C\smallsetminus(-\infty,0];\label{eq:Landen_1}\\\mathbf K(\sqrt{\lambda})={}&\frac{1}{1+\sqrt{\lambda}}\mathbf K\left( \frac{2\sqrt[4]{\lambda}}{1+\sqrt{\lambda}} \right),\quad |\lambda|<1,\label{eq:Landen_2}
\end{align}  we obtain the degree-2 transformations of the modular lambda function (see \cite[][\S135]{WeberVol3})\begin{align}\label{eq:double_half_lambda}\lambda(2z)=\left[ \frac{1-\sqrt{1-\lambda(z)}}{1+\sqrt{1-\lambda(z)}} \right]^2,\quad \lambda\left( \frac{\vphantom{1}z}{2} \right)=\frac{4\sqrt{\lambda(z)}}{[1+\sqrt{\lambda(z)}]^2},\end{align}applicable to $ z\in\Int\mathfrak D_\varLambda$.

\begin{proof}[Proof of Theorem \ref{thm:int_repn_Epstein_zeta}]We will only work out the details for Eq.~\ref{eq:K_int_repn2}, as the computations for Eqs.~\ref{eq:K_int_repn3}--\ref{eq:K_int_repn7} are essentially similar.

To prove   Eq.~\ref{eq:K_int_repn2}, it would suffice to establish the following identities for $ z/i>0$:  \begin{align}\sum_{n=0}^\infty\frac{1}{(2n+1)^{3}}\frac{1}{e^{(2n+1)\pi \frac{2z}{i}}-1}={}&-\int_0^{\lambda(z)}\frac{[\mathbf K(\sqrt{\smash[b]{t}})]^2}{16}\left[ \frac{i\mathbf K(\sqrt{1-t})}{\mathbf K(\sqrt{t})} -z\right]^2\frac{t\D t}{1-t}\notag\\={}&-\int_0^{\lambda(2z)}\frac{[\mathbf K(\sqrt{\smash[b]s})]^2}{8}\left[ \frac{i\mathbf K(\sqrt{1-s})}{\mathbf K(\sqrt{s})} -2z\right]^2\frac{\D s}{1-s},
\label{eq:E-3_E4_Eichler}\\
\sum_{n=0}^\infty\frac{1}{(2n+1)^{3}}\frac{1}{e^{(2n+1)\pi \frac{2z}{i}}+1}={}&-\int_0^{\lambda(2z)}\frac{[\mathbf K(\sqrt{\smash[b]t})]^2}{8}\left[ \frac{i\mathbf K(\sqrt{1-t})}{\mathbf K(\sqrt{t})} -2z\right]^2\D t.\label{eq:sum_Eichler_a}
\end{align}The verification of the aforementioned connections between series and integrals will occupy the rest of this proof.

If we set $ \sigma_k(n)=\sum_{d\mid n}d^k$, then we have $ \sum_{n=1}^\infty\frac{1}{n^3(e^{-2\pi inz}-1)}=\sum_{n=1}^\infty\frac{\sigma_3(n)}{n^3}e^{2\pi inz},z\in\mathfrak H$, and  the weight-4 Eisenstein series satisfies   $ E_4(\zeta):=1-240\sum_{n=1}^\infty\frac{n^3e^{2\pi in\zeta}}{1-e^{2\pi in\zeta}}=1-240\sum_{n=1}^\infty\sigma_3(n)e^{2\pi in\zeta},\zeta\in\mathfrak H$. Thus, we have \begin{subequations}\begin{align}
\sum_{n=1}^\infty\frac{1}{n^3(e^{-2\pi inz}-1)}={}&\frac{(2\pi i)^3}{480}\int_z^{i\infty}[1-E_4(\zeta)](\zeta-z)^2\D\zeta,\label{eq:Eichler_E4a}\\\sum_{n=1}^\infty\frac{1}{(2n)^3(e^{-4\pi inz}-1)}={}&\frac{(2\pi i)^3}{480}\int_z^{i\infty}[1-E_4(2\zeta)](\zeta-z)^2\D\zeta.\label{eq:Eichler_E4b}
\end{align} \end{subequations}(This argument is a standard trick in treating Eichler integrals. See, for example, \cite[][\S1]{GunMurtyRath2011}.)

Next, we use a variable substitution $ t=\lambda(\zeta)\in(0,1)$ for $ \zeta/i>0$, where $ \lambda(\cdot) $ is the modular lambda function, and (cf.~\cite[][Eqs.~2.3.32, 2.3.34, 2.3.25]{AGF_PartI} as well as Eqs.~\ref{eq:lambda_K_rln}, \ref{eq:Landen_1} and \ref{eq:double_half_lambda} given above){\allowdisplaybreaks\begin{align}
E_4(\zeta)={}&\left[ \frac{2\mathbf K(\sqrt{\smash[b]{\lambda(\zeta)}})}{\pi} \right]^{4}\{1-\lambda(\zeta)+[\lambda(\zeta)]^{2}\},\\E_4(2\zeta)={}&\left\{ \frac{[1+\sqrt{1-\lambda(\smash[b]{\zeta})}]\mathbf K(\sqrt{\smash[b]{\lambda(\zeta)}})}{\pi} \right\}^{4}\left\{1-\left[ \frac{1-\sqrt{1-\lambda(\smash[b]{\zeta})}}{1+\sqrt{1-\lambda(\smash[b]{\zeta})}} \right]^{2}+\left[ \frac{1-\sqrt{1-\lambda(\smash[b]{\zeta})}}{1+\sqrt{1-\lambda(\smash[b]{\zeta})}} \right]^{4}\right\},\\\zeta={}&\frac{i\mathbf K(\sqrt{1-\lambda(\smash[b]{\zeta})})}{\mathbf K(\sqrt{\lambda(\smash[b]{\zeta})})},\\\frac{\D}{\D t}\frac{\mathbf K(\sqrt{1-t})}{\mathbf K(\sqrt{t})}={}&-\frac{\pi}{4t(1-t)[\mathbf K(\sqrt{t})]^2}.
\end{align}}This allows us to recast the Eichler integrals over Eisenstein series (the right-hand sides of Eqs.~\ref{eq:Eichler_E4a}--\ref{eq:Eichler_E4b}) into  integrals whose integrands involve the products of two complete elliptic integrals of the first kind. Subsequently, the variable substitutions give rise to\begin{align}&
\int_z^{i\infty}[E_{4}(2\zeta)-E_4(\zeta)](\zeta-z)^2\D\zeta\notag\\={}&\frac{\pi i}{4}\int_0^{\lambda(z)}\left[ \frac{2\mathbf K(\sqrt{\smash[b]{t}})}{\pi} \right]^{4}\left[\left(1-t+\frac{t^{2}}{16}\right)-(1-t+t^{2})\right] \left[ \frac{i\mathbf K(\sqrt{1-t})}{\mathbf K(\sqrt{t})} -z\right]^2\frac{\D t}{t(1-t)[\mathbf K(\sqrt{t})]^2}.
\end{align} This proves the first equality in Eq.~\ref{eq:E-3_E4_Eichler}.

For  the second equality in Eq.~\ref{eq:E-3_E4_Eichler}, one uses a variable substitution $ t=\lambda(\zeta)=4\sqrt{\lambda(2\zeta)}/[1+\sqrt{\lambda(2\zeta)}]^2=4\sqrt{s}/(1+\sqrt{s})^2$ (Eq.~\ref{eq:double_half_lambda}) and appeals to Landen's transformation (Eqs.~\ref{eq:Landen_1} and \ref{eq:Landen_2}).

To prove Eq.~\ref{eq:sum_Eichler_a}, we write its left-hand side as\begin{align}
\sum_{n=0}^\infty\frac{1}{(2n+1)^{3}}\left[\frac{1}{e^{(2n+1)\pi \frac{2z}{i}}-1}-\frac{2}{e^{(2n+1)\pi \frac{4z}{i}}-1}\right].
\end{align} By  Eq.~\ref{eq:E-3_E4_Eichler}, such an infinite sum is equal to \begin{align}
-\int_0^{\lambda(2z)}\frac{[\mathbf K(\sqrt{\smash[b]s})]^2}{8}\left[ \frac{i\mathbf K(\sqrt{1-s})}{\mathbf K(\sqrt{s})} -2z\right]^2\frac{\D s}{1-s}+\int_0^{\lambda(2z)}\frac{[\mathbf K(\sqrt{\smash[b]{t}})]^2}{8}\left[ \frac{i\mathbf K(\sqrt{1-t})}{\mathbf K(\sqrt{t})} -2z\right]^2\frac{t\D t}{1-t},
\end{align}which can be identified with the right-hand side of  Eq.~\ref{eq:sum_Eichler_a}.
\end{proof}\begin{remark}For $ s\in\{2,3,4,5,7\} $ where there are no cusp forms of weight $ 2s$ on $ \varGamma_0(1)=SL(2,\mathbb Z)$, we have constructed integral representations for the automorphic Green's functions $ G_s^{\mathfrak H/\overline{\varGamma}_0(1)}(z,z')$, and hence for the Epstein zeta functions   $ E^{\varGamma_0(1)}(z,s)=\frac{1-2s}{4\pi}\lim_{ z'\to i\infty}G_s^{\mathfrak H/\overline{\varGamma}_0(1)}(z,z')(\I z')^{s-1}$, in \cite[][Propositions 2.1.2 and 2.3.2]{AGF_PartI}. Interested readers may wish to  check the numerical consistency between our formulae for $ E^{\varGamma_0(1)}(z,s),s\in\{2,3,4,5,7\}$ in \cite{AGF_PartI} and the integral representations for   $ E^{\varGamma_0(4)}(z,s),s\in\{2,3,4,5,7\}$  stated in  Theorem \ref{thm:int_repn_Epstein_zeta}. \eor\end{remark}\begin{remark}If we formally rewrite the integral representations in Eqs.~\ref{eq:K_int_repn2}--\ref{eq:K_int_repn7} as $ E^{\varGamma_0(4)}(-1/(4z),s)=\varPsi(2z+1)$, and spell out the symmetry $ \varPsi(2z+1)=\varPsi(-1/(2z+1))$ (cf.~Eq.~\ref{eq:lambda_inv} as well as the proof of Theorem~\ref{thm:Ramanujan_series}\ref{itm:Epstein_Hecke4}) in terms of integrals over products of complete elliptic integrals of the first kind, then we obtain some  integral representations of odd zeta values:{\allowdisplaybreaks\begin{align}
\zeta(3)={}&\frac{2}{7}\int_0^1[\mathbf K(\sqrt{1-t})]^2\D t,\label{eq:zeta3_int}\\\zeta(5)={}&\frac{8}{93}\int_0^1(1-2t)[\mathbf K(\sqrt{1-t})]^4\D t,\label{eq:zeta5_int}\\\zeta(7)={}&\frac{32}{5715}\int_0^1[2-17t(1-t)][\mathbf K(\sqrt{1-t})]^6\D t,\label{eq:zeta7_int}\\\zeta(9)={}&\frac{128}{160965}\int_0^1(1-2t)[1-31t(1-t)][\mathbf K(\sqrt{1-t})]^8\D t,\label{eq:zeta9_int}\\\zeta(13)={}&\frac{4096}{3831545025}\int_0^1(1-2t)[1-512t(1-t)+5461 t^2 (1-t)^2][\mathbf K(\sqrt{1-t})]^{12}\D t,\label{eq:zeta13_int}
\end{align}as well as some  vanishing identities:\begin{align}
0={}&\int_0^1[2-17t(1-t)][\mathbf K(\sqrt{1-t})]^2[\mathbf K(\sqrt{t})]^{4}\D t,\label{eq:vanish_int1}\\0={}&\int_0^1(1-2t)[1-31t(1-t)][\mathbf K(\sqrt{1-t})]^2[\mathbf K(\sqrt{t})]^{6}\D t,\label{eq:vanish_int2}\\0={}&\int_0^1(1-2t)[1-512t(1-t)+5461 t^2 (1-t)^2][\mathbf K(\sqrt{1-t})]^{2}[\mathbf K(\sqrt{t})]^{10}\D t,\label{eq:vanish_int3}\\0={}&\int_0^1(1-2t)[1-512t(1-t)+5461 t^2 (1-t)^2][\mathbf K(\sqrt{1-t})]^{4}[\mathbf K(\sqrt{t})]^{8}\D t.\label{eq:vanish_int4}
\end{align}Here, Eq.~\ref{eq:zeta3_int} is well known (cf.~\cite[][item 7.112.3]{GradshteynRyzhik}),  Eq.~\ref{eq:zeta5_int} appeared at the end of \cite{Zhou2013Pnu} as well as in \cite[][Eq.~42]{WanZucker2014}, while Eqs.~\ref{eq:zeta7_int} and \ref{eq:zeta9_int} are special cases of \cite[][Eqs.~43 and 44]{WanZucker2014}; the vanishing identities in Eqs.~\ref{eq:vanish_int1}--\ref{eq:vanish_int4} can be explained by certain trivial zeros of the Riemann zeta function \cite[][Eq.~33]{WanZucker2014}. If we set $ 2z+1=i$ in Eqs.~\ref{eq:K_int_repn2}--\ref{eq:K_int_repn7}, then we can represent (cf.~Eqs.~\ref{eq:Epstein_SL2Z_i_s}, \ref{eq:Epstein_Hecke2_add_form}, \ref{eq:Epstein_Hecke2_SL2Z_spec_pt})\begin{align} E^{\varGamma_0(4)}\left( \frac{1+i}{4} ,s\right)=\frac{E^{\varGamma_0(1)}(i,s)}{2^{s}(2^s+1)}=\frac{2^{1-s}\zeta(s)L(s,\chi_{-4})}{(2^{s}+1)\zeta(2s)}\end{align}as integrals over $ \mathbf K$:\begin{align}
\frac{3G}{2\pi^{2}}={}&\frac{21\zeta(3)}{4\pi^{3}}-\frac{3}{2\pi^{3}}\int_0^{1/2}[\mathbf K(\sqrt{t})]^2\left\{\left[ \frac{\mathbf K(\sqrt{1-t})}{\mathbf K(\sqrt{t})} \right]^{2}-1\right\}\D t,\label{eq:G_zeta3_K_int}\\\frac{105 \zeta (3)}{128 \pi ^3}={}&\frac{1395 \zeta (5)}{64 \pi ^5}+\frac{15}{8 \pi ^5}\int_0^{1/2}(2t-1)[\mathbf K(\sqrt{t})]^4\left\{\left[ \frac{\mathbf K(\sqrt{1-t})}{\mathbf K(\sqrt{t})} \right]^{2}-1\right\}^{2}\D t,\\\frac{105L(4,\chi_{-4})}{136 \pi ^4}={}&\frac{200025 \zeta (7)}{2176 \pi ^7}-\frac{70}{136 \pi ^7}\int_0^{1/2}[2-17t(1-t)][\mathbf K(\sqrt{t})]^6\left\{\left[ \frac{\mathbf K(\sqrt{1-t})}{\mathbf K(\sqrt{t})} \right]^{2}-1\right\}^{3}\D t,\\\frac{4725 \zeta (5)}{8192 \pi ^5}={}&\frac{50703975 \zeta (9)}{126976 \pi ^9}+\frac{315}{992 \pi ^9}\int_0^{1/2}(2t-1)[1-31t(1-t)][\mathbf K(\sqrt{t})]^8\left\{\left[ \frac{\mathbf K(\sqrt{1-t})}{\mathbf K(\sqrt{t})} \right]^{2}-1\right\}^{4}\D t,\\\frac{8243235 \zeta (7)}{22544384 \pi ^7}={}&\frac{11506129710075 \zeta (13)}{1431568384 \pi ^{13}}+\frac{3003}{349504 \pi ^{13}}\int_0^{1/2}(2t-1)[1-512t(1-t)+5461 t^2 (1-t)^2]\times\notag\\{}&\times[\mathbf K(\sqrt{t})]^{12}\left\{\left[ \frac{\mathbf K(\sqrt{1-t})}{\mathbf K(\sqrt{t})} \right]^{2}-1\right\}^{6}\D t.
\end{align}}Here, the result in Eq.~\ref{eq:G_zeta3_K_int} can be deduced from the  Legendre differential equations (cf.~\cite[][items 7.112.3 and 7.112.5]{GradshteynRyzhik}).  \eor\end{remark}

\begin{remark}As we are mainly concerned with the Eichler integrals related to automorphic Green's functions  $ G_s^{\mathfrak H/PSL(2,\mathbb Z)}(z,z')$ when there are no cusp forms of weight $ 2s$ on $ SL(2,\mathbb Z)$, we have only displayed results for $s\in\{2,3,4,5,7\} $ in Eqs.~\ref{eq:zeta3_int}--\ref{eq:zeta13_int}.  However, the methods for converting Ramanujan series to integral forms are not necessarily limited to such special values of $s$, so the integral representations for other odd zeta values can still be constructed in a similar vein.

In fact, as recently pointed out by Wan and Zucker \cite[][Theorem 1]{WanZucker2014}, for each positive integer $ n$, there exists a polynomial function $ f_n(t)$ with rational coefficients, satisfying $ f_n(t)=(-1)^{n+1}f_{n}(1-t)$, $ \deg f_n\leq n-1$ and $\frac{1}{\zeta(2n+1)}\int_0^1f_n(t)[\mathbf K(\sqrt t)]^{2n}\D t\in\mathbb Q.$ However, we note that for positive integers $ n\notin\{1,2,3,4,6\}$, where there exist cusp forms of weights $ 2n+2$ on  $ SL(2,\mathbb Z)$,  the polynomials  $ f_n(t)$ meeting the aforementioned qualifications are not necessarily unique up to a multiplicative constant. For example, one can superimpose any rational multiples of a vanishing identity $\int_0^1[8-1049t(1-t)][\mathbf K(\sqrt t)]^{10}\D t=0$ onto a formula of Wan and Zucker \cite[][Eq.~45]{WanZucker2014} $ \zeta(11)=\frac{512}{29016225}\int_0^1[2-259t(1-t)+1382t^{2}(1-t)^2][\mathbf K(\sqrt t)]^{10}\D t$ to obtain alternative integral representations of $ \zeta(11)$. Such a lack of uniqueness in certain integral representations of odd zeta values can be systematically  explained in the language of newforms, which we hope to address in a separate article.    \eor\end{remark}

\bibliography{Epstein}

\begin{thebibliography}{10}

\bibitem{RLN4}
{\sc Andrews, G.~E., and Berndt, B.~C.}
\newblock {\em Ramanujan's Lost Notebook (Part IV)}.
\newblock Springer-Verlag, New York, NY, 2013.

\bibitem{ApostolVol2}
{\sc Apostol, T.~M.}
\newblock {\em Modular Functions and Dirichlet Series in Number Theory},
  vol.~41 of {\em Graduate Texts in Mathematics}.
\newblock Springer-Verlag, New York, NY, 1976.

\bibitem{ET1}
{\sc Bateman, H.}
\newblock {\em Table of Integral Transforms}, vol.~I.
\newblock McGraw-Hill, New York, NY, 1954.
\newblock (compiled by staff of the Bateman Manuscript Project: Arthur
  Erd{\'e}lyi, Wilhelm Magnus, Fritz Oberhettinger, Francesco G. Tricomi, David
  Bertin, W. B. Fulks, A. R. Harvey, D. L. Thomsen, Jr., Maria A. Weber and E.
  L. Whitney).

\bibitem{Berndt1978}
{\sc Berndt, B.~C.}
\newblock Analytic {E}isenstein series, theta-functinos, and series relations
  in the spirit of {R}amanujan.
\newblock {\em J. Reine Angew. Math. 304\/} (1978), 332--365.

\bibitem{RN2}
{\sc Berndt, B.~C.}
\newblock {\em Ramanujan's Notebooks (Part II)}.
\newblock Springer-Verlag, New York, NY, 1989.

\bibitem{Berndt1992}
{\sc Berndt, B.~C.}
\newblock On a certain theta-function in a letter of {R}amanujan from {F}itzroy
  {H}ouse.
\newblock {\em Ganita 43\/} (1992), 33--43.

\bibitem{BersJohnSchechter1964}
{\sc Bers, L., John, F., and Schechter, M.}
\newblock {\em Partial Differential Equations}, vol.~III of {\em Lectures in
  Applied Mathematics (Proceedings of the Summer Seminar, Boulder, Colorado,
  1957)}.
\newblock Interscience Publishers, New York, NY, 1964.

\bibitem{GradshteynRyzhik}
{\sc Gradshteyn, I.~S., and Ryzhik, I.~M.}
\newblock {\em Table of Integrals, Series, and Products}, 7th~ed.
\newblock Academic Press, Burlington, MA, 2007.
\newblock (Translated from Russian by Scripta Technica, Inc., edited by Alan
  Jeffrey and Daniel Zwillinger).

\bibitem{GrossZagierII}
{\sc Gross, B., Kohnen, W., and Zagier, D.}
\newblock Heegner points and derivatives of ${L}$-series. {II}.
\newblock {\em Math.\ Ann. 278\/} (1987), 497--562.

\bibitem{GrossZagier1985}
{\sc Gross, B.~H., and Zagier, D.~B.}
\newblock On singular moduli.
\newblock {\em J.\ Reine Angew.\ Math. 355\/} (1985), 191--220.

\bibitem{GrossZagierI}
{\sc Gross, B.~H., and Zagier, D.~B.}
\newblock Heegner points and derivatives of ${L}$-series.
\newblock {\em Invent.\ Math. 84\/} (1986), 225--320.

\bibitem{GunMurtyRath2011}
{\sc Gun, S., Murty, M.~R., and Rath, P.}
\newblock Transcendental values of certain {E}ichler integrals.
\newblock {\em Bull. London Math. Soc. 43\/} (2011), 939--962.

\bibitem{Hejhal1983}
{\sc Hejhal, D.~A.}
\newblock {\em The {S}elberg Trace Formula for ${PSL}(2,\mathbb{R})$ (Volume
  2)}, vol.~1001 of {\em Lecture Notes in Mathematics}.
\newblock Springer-Verlag, 1983.

\bibitem{WanZucker2014}
{\sc Wan, J.~G., and Zucker, I.~J.}
\newblock Integrals of ${K}$ and ${E}$ from lattice sums.
\newblock \texttt{arXiv:1410.7081v1} [math.NT], 2014.

\bibitem{Watson1995Bessel}
{\sc Watson, G.~N.}
\newblock {\em A Treatise on the Theory of {B}essel Functions}, 2nd~ed.
\newblock Cambridge Mathematical Library. Cambridge University Press,
  Cambridge, UK, 1995.

\bibitem{WeberVol3}
{\sc Weber, H.}
\newblock {\em \selectlanguage{german}{L}ehrbuch der {A}lgebra. \emph{Bd. 3:}
  {E}lliptische {F}unktionen und algebraische
  {Z}ahlen\selectlanguage{english}}.
\newblock Friedrich Vieweg und Sohn, Braunschweig, Germany, 1908.
\newblock (reprinted as \selectlanguage{german}\textit{{L}ehrbuch der
  {A}lgebra}\selectlanguage{english} Vol.III by Chelsea, NY, 1961).

\bibitem{Williams1999}
{\sc Williams, K.~S.}
\newblock Some {L}ambert series expansions of products of theta functions.
\newblock {\em Ramanujan J. 3\/} (1999), 367--384.

\bibitem{AGF_PartI}
{\sc Zhou, Y.}
\newblock {K}ontsevich--{Z}agier integrals for automorphic {G}reen's functions.
  {I}.
\newblock {\em Ramanujan J.\/}.
\newblock \url{doi:10.1007/s11139-014-9663-7} (to appear) {\tt
  arXiv:1312.6352v4} [math.CA].

\bibitem{Zhou2013Pnu}
{\sc Zhou, Y.}
\newblock Legendre functions, spherical rotations, and multiple elliptic
  integrals.
\newblock {\em Ramanujan J. 34\/} (2014), 373--428.

\end{thebibliography}
\bibliographystyle{acm}
\end{document}